\documentclass[reqno,b5paper]{amsart}

\usepackage{amsmath}
\usepackage{amssymb}
\usepackage{amsthm}
\usepackage{enumerate}
\usepackage[mathscr]{eucal}
\usepackage{eqlist}

\theoremstyle{plain}
\newtheorem{theorem}{Theorem}[section]
\newtheorem{prop}[theorem]{Proposition}
\newtheorem{lemma}[theorem]{Lemma}
\newtheorem{corol}[theorem]{Corollary}

\newtheorem{definition}[theorem]{Definition}

\theoremstyle{remark}
\newtheorem{example}{Example}

\numberwithin{equation}{section}

\begin{document}
\title[The Closure Operator of $es$-Splitting Matroids]%
{The Closure Operator of $es$-Splitting Matroids}
\author[S. B. Dhotre, P. P. Malavadkar \and M. M. Shikare]%
{S. B. Dhotre$^{1}$, P. P. Malavadkar$^{2}$  \and M. M. Shikare$^{3}$}

\newcommand{\acr}{\newline\indent}

\address{\llap{1\,}Department of Mathematics\acr
                   Savitribai Phule Pune University\acr
                   Pune-411007, India}
                 
\email{dsantosh2@yahoo.co.in}

\address{\llap{2\,}MIT College of Engineering\acr
                   Pune-411038, India}
                 
\email{pmalavadkar@gmail.com}

\address{\llap{3\,}Department of Mathematics\acr
                   Savitribai Phule Pune University\acr
                   Pune-411007, India}
\email{mmshikare@gmail.com}



\begin{abstract}
The es-splitting operation for binary matroids is a natural generalization of Slater's $n$-line splitting operation on graphs. In this paper, we characterize the closure operator of the es-splitting binary matroid $M^e_X$ in terms of the closure operator of the original binary matroid $M$. We also characterize the flats of the es-splitting binary matroid $M^e_X$ in terms of the flats of the original binary matroid $M$.
\end{abstract}

\maketitle

\section{Introduction} 
The es-splitting operation for binary matroids is a natural generalization of $n$-line splitting operation on graphs. The $n$-line splitting operation on graphs is specified by Slater \cite{SlaterC4CG} as follows.

Let $G$ be a graph and $e=uv$ be an edge of $G$ with deg $u\geq 2n-3$ with $u$ adjacent to $v,~x_{1},~x_{2},~...,~x_{k},~y_{1},~y_{2},~...,~y_{h},$ where $k$ and $h\geq n-2.$ Let $H$ be the graph obtained from $G$ by replacing $u$ by two adjacent vertices $u_{1}$ and $u_{2},$ with $v$ adj $u_{1},~v$ adj $u_{2},~u_{1}$ adj $x_{i}~(1\leq i\leq k),$ and $u_{2}$ adj $y_{j}~(1\leq j\leq h),$ where deg $u_{1}\geq n$ and deg $u_{2}\geq n.$ The transition from $G$ to $H$ is called an {\it $n$-line splitting operation}. In other words, we say that $H$ is an $n$-line splitting of $G.$ This construction is explicitly illustrated with the help of Figure 1.

\unitlength=0.7mm \special{em:linewidth 0.4pt} \linethickness{0.4pt}\begin{picture}(140.67,48.00)(10,0)
\put(35.00,20.00){\circle*{1.33}}
\put(70.00,20.00){\circle*{1.33}}
\put(70.00,42.33){\circle*{1.33}}
\put(35.00,42.33){\circle*{1.33}}
\put(52.67,30.67){\circle*{1.33}}
\put(52.67,15.00){\circle*{1.33}}
\put(35.00,35.00){\circle*{1.33}}
\put(70.00,35.00){\circle*{1.33}}
\put(52.67,30.67){\line(-3,2){17.67}}
\put(52.67,30.67){\line(-4,1){17.67}}
\put(52.67,30.67){\line(-5,-3){17.67}}
\put(52.67,15.00){\line(0,1){15.67}}
\put(52.67,30.67){\line(3,2){17.33}}
\put(52.67,30.33){\line(4,1){17.33}}
\put(52.67,30.67){\line(5,-3){17.33}}
\put(52.67,34.33){\makebox(0,0)[cc]{$u$}}
\put(52.67,11.67){\makebox(0,0)[cc]{$v$}}
\put(30.67,42.67){\makebox(0,0)[cc]{$x_{1}$}}
\put(30.67,34.67){\makebox(0,0)[cc]{$x_{2}$}}
\put(30.67,19.33){\makebox(0,0)[cc]{$x_{k}$}}
\put(74.33,19.33){\makebox(0,0)[cc]{$y_{h}$}}
\put(74.33,34.67){\makebox(0,0)[cc]{$y_{2}$}}
\put(74.33,43.00){\makebox(0,0)[cc]{$y_{1}$}}
\put(50.33,22.67){\makebox(0,0)[cc]{$e$}}
\put(38.33,31.67){\circle*{0.67}}
\put(38.33,28.67){\circle*{0.67}}
\put(38.33,25.67){\circle*{0.67}}
\put(65.33,31.00){\circle*{0.67}}
\put(65.33,28.33){\circle*{0.67}}
\put(65.33,26.00){\circle*{0.67}}
\put(90.00,20.00){\circle*{1.33}}
\put(136.67,20.00){\circle*{1.33}}
\put(136.67,35.00){\circle*{1.33}}
\put(136.67,42.33){\circle*{1.33}}
\put(90.00,42.33){\circle*{1.33}}
\put(90.00,35.00){\circle*{1.33}}
\put(113.67,15.00){\circle*{1.33}}
\put(102.00,35.00){\circle*{1.33}}
\put(125.33,35.00){\circle*{1.33}}
\put(90.00,35.00){\line(1,0){46.67}}
\put(125.33,35.00){\line(3,2){11.33}}
\put(125.33,35.00){\line(3,-4){11.33}}
\put(125.33,35.00){\line(-3,-5){12.00}}
\put(113.67,15.00){\line(-3,5){12.00}}
\put(102.00,35.00){\line(-5,3){12.00}}
\put(102.00,35.00){\line(-4,-5){12.00}}
\put(102.00,39.00){\makebox(0,0)[cc]{$u_{1}$}}
\put(125.33,39.00){\makebox(0,0)[cc]{$u_{2}$}}
\put(85.67,42.33){\makebox(0,0)[cc]{$x_{1}$}}
\put(85.67,35.00){\makebox(0,0)[cc]{$x_{2}$}}
\put(85.67,20.00){\makebox(0,0)[cc]{$x_{k}$}}
\put(113.67,12.33){\makebox(0,0)[cc]{$v$}}
\put(140.67,42.33){\makebox(0,0)[cc]{$y_{1}$}}
\put(140.67,35.00){\makebox(0,0)[cc]{$y_{2}$}}
\put(140.67,20.00){\makebox(0,0)[cc]{$y_{h}$}}
\put(113.67,37.67){\makebox(0,0)[cc]{$a$}}
\put(105.00,24.33){\makebox(0,0)[cc]{$e$}}
\put(122.00,24.33){\makebox(0,0)[cc]{$\gamma$}}
\put(93.67,32.67){\circle*{0.67}}
\put(93.67,30.33){\circle*{0.67}}
\put(93.67,28.00){\circle*{0.67}}
\put(132.67,33.33){\circle*{0.67}}
\put(132.67,31.33){\circle*{0.67}}
\put(132.67,29.00){\circle*{0.67}}
\put(52.67,5.67){\makebox(0,0)[cc]{$G$}}
\put(113.67,5.67){\makebox(0,0)[cc]{$H$}}
\put(79.67,3.67){\makebox(0,0)[cc]{\bf {Figure 1}}}
\end{picture}
\vskip.2cm\noindent

Slater \cite{SlaterC4CG} characterized 4-connected graphs in terms of the 4-line splitting operation along with some other operations.

We study the relation between incidence matrices of the graphs $G$ and $H$, respectively. Suppose $G$ is a graph with $n$ vertices and $m$ edges. Let $X = \{e, x_1, x_2, \cdots, x_k\}$ be a subset of set of edges incident at $u$. The incident matrix $A$ of $G$ is a matrix of size $n \times m$ whose rows correspond to vertices and columns correspond to edges. The row corresponding to the vertex $u$ has $1$'s in the columns of $e, x_1, x_2, \cdots, x_k,$ $y_1, y_2, \cdots, y_h$ and $0$ in the other columns. The graph $H$ has $(n+1)$ vertices and $(m+2)$ edges. The incidence matrix $A'$ of $H$ is a matrix of size $(n+1) \times (m+2)$. The row corresponding to $u_2$ has $1$ in the columns of $y_1, y_2, \cdots, y_h, \gamma, a$ and $0$ in other columns, where as the row corresponding to the vertex $u_1$ has $1$ in the columns of $e, x_1, x_2, \cdots, x_k, a$ and $0$ in other columns. 

	One can see that the matrix $A'$ can be obtained from $A$ by adjoining an extra row (corresponding to the vertex $u_1$) to $A$ with entries zero every where except in the columns corresponding to $e, x_1, x_2, \cdots, x_k$ where it takes the value $1$. The row vector obtained by addition(mod 2) of row vectors corresponding to vertices $u$ and $u_1$ will corresponds to the row vector of the vertex $u_2$ and adjoining two columns labelled $a$ and $\gamma$ to the resulting matrix such that the column labelled a is zero everywhere except in the row corresponding to $u_1$ where it takes the value 1, and $\gamma$ is sum of the two column vectors corresponding to the elements $a$ and $e$ in $A'$. 

Noticing the above fact, Shikare and Azanchiler \cite{Habib, HES} extended the notion of $n$-line-splitting operation from graphs to binary matroids in the following way:
\begin{definition}\label{dfes} Let $M$ be a binary matroid on a set $E$ and let $X$ be a subset of $E$ with $e\in X$. Suppose $A$ is a matrix representation of $M$ over GF(2). Let $A^e_X$ be a matrix obtained from $A$ by adjoining an extra row $\delta_X$ to $A$ with entries zero every where except in the columns corresponding to the elements of $X$ where it takes the value 1 and then adjoining two columns labelled $a$ and $\gamma$ to the resulting matrix such that the column labelled a is zero everywhere except in the last row where it takes the value 1, and $\gamma$ is the sum of the two column vectors corresponding to the elements $a$ and $e$. The vector matroid of the matrix $A_X^e$ is denoted by $M_X^e$. The  transition from $M$ to $M_X^e$ is called an es-splitting operation. We call the matroid $M_X^e$ as es-splitting matroid. \end{definition} 
If $|X| = 2$ and $X = \{x, y\} \subseteq E(M)$, then we denote the matroid $M^e_X$ by $M^e_{x, y}$. 

\section{Properties of es-splitting matroids}
Several properties concerning es-splitting operation have been explored in \cite{HES, HChB, DMSES3C, MDSCESM, SDMRESCG}. The following proposition characterizes the circuits of the matroid  $M_X^e$ in terms of the circuits of the matroid $M$(see \cite{HES}).
\begin{prop}\label{CMXE} Let $M(E, \mathcal C)$ be a binary matroid on $E$ together with the collection of circuits $\mathcal C$. Suppose $X\subseteq E$, $e\in X$ and $a, \gamma \notin E$. Then $M_X^e = (E \cup \{a, \gamma\}, \mathcal C')$ where $\mathcal C' = \mathcal {C}$$_0$ $ \cup$ $ \mathcal C$$_1$ $ \cup $ $\mathcal C$$_2$ $ \cup $ $\mathcal C$$_3$ $ \cup$ $ \{\Delta \} $ with $\Delta = \{e, a, \gamma \}$ and \\
\begin{eqnarray}
\mbox{$\mathcal C$$_0$} &=& \mbox{$\{ C\in $ $\mathcal C$$ ~|~ C$ contains an even number of elements of $X$ $\}$;}\nonumber\\
\mbox{$\mathcal C$$_1$} &=& \mbox{The set of minimal members of $\{ C_1\cup C_2 ~|~ C_1, C_2\in $ $ \mathcal C$$, C_1\cap C_2 = \phi $}\nonumber\\
&&\mbox{ and each of $C_1$ and $C_2$ contains an odd number of elements of} \nonumber\\
&&\mbox{$X$ such that $C_1\cup C_2$ contains no member of $\mathcal C$$_0$ $\}$;}\nonumber\\
\mbox{$\mathcal C$$_2$} &=& \mbox{$\{C\cup \{a\}~|~C\in$ $\mathcal C$ and $C$ contains an odd number of elements of $X \}$;}\nonumber\\
\mbox{$\mathcal C$$_3$} &=& \mbox{$\{C\cup \{e, \gamma\}~|~ C\in $ $ \mathcal C$$, e\notin C$ and C contains an odd number of elements}\nonumber\\ 
&&\mbox{of X $\}$ $\cup $ $\{(C\setminus e)\cup \{\gamma\}~|~ C\in $ $ \mathcal C$$, e\in C$ and C contains an odd number}\nonumber\\
&&\mbox{of elements of $X$ $\}$ $\cup $ $\{(C\setminus e)\cup \{a, \gamma\}~|~ C\in $ $ \mathcal C$$, e\in C$ and $C\setminus e$}\nonumber\\
&&\mbox{contains an odd number of elements of $X$ $\}$.}\nonumber \end{eqnarray}
\end{prop} 

We denote by $\mathcal C$$_{OX}$ the set of all circuits of the matroid $M$ each of which contains an odd number of elements of the set $X\subseteq E(M)$. The members of the set $\mathcal C$$_{OX}$ are called \textit{OX-circuits}. On the other hand, $\mathcal C$$_{EX}$ denotes the set of all circuits of $M$ each of which contains an even number of elements of the set $X$. The members of the set $\mathcal C$$_{EX}$ are called \textit{EX-circuits}. \\

Let $M$ be a matroid on $E$ with rank function $r.$ Then the function  $cl$ from $2^{E}$ into $2^{E}$ defined, for all $A\subseteq E,$ by $cl(A)=\{x\in E~|~r(A\cup  x)=r(A)\},$ where $2^E$ denote the power set of $E,$ is called the {\it closure operator} of $M.$


\begin{prop} \cite{Oxley} Let $M$ be a matroid on $E$ and $A\subseteq E.$ Then\\
$cl(A)=A\cup\{x\in E~|~M$ has a circuit $C$ such that $x\in C\subseteq A\cup  x\}.$\end{prop}

\begin{lemma}\cite{Oxley} Suppose $M$ is a matroid on $E$ with the rank function $r,$ $A\subseteq E$ and $x\in cl(A)$. Then $cl(A\cup x) = cl(A)$.\end{lemma}

Let $cl$ and $cl'$ be the {\it closure operators} of $M$ and $M_X^e$, respectively. The following result characterizes the rank function of the matroid $M_X^e$ in terms of the rank function of the matroid $M$ (see \cite{HES, DMSES3C}).
\begin{lemma}\label{RMXE} Let $r$ and $r'$ be the rank functions of the matroids $M$ and  $M_X^e$, respectively. Suppose that $A \subseteq E(M)$. Then 
\begin{itemize}
\item[(1)] $ r'(A) = r(A) + 1,$ if $A$ contains an $OX$-circuit of the matroids $M$;\\
 $~~~~~~~~~~~~ = r(A)$; otherwise.
\item[(2)] $ r'(A\cup a) = r(A) + 1$;
\item[(3)] $ r'(A\cup \{\gamma\}) = r(A)$ if not $A$ but $A\cup \{e\}$ contains an $OX$-circuit of $M$;\\
$~~~~~~~~~~~~~~ ~~~~~~~= r(A) + 2$; if $A$ contains an $OX$-circuit of $M$ and $e \notin cl(A);$\\ 
$~~~~~~~~~~~~~~ ~~~~~~~= r(A) + 1$; otherwise 
\item[(4)] $r'(A\cup \{a, \gamma \}) = r(A) + 1 $ if $ e\in cl(A);$\\
$~~~~~~~~~~~~~~~~~~~~~~~~ = r(A) + 2 $ if $e \notin cl(A);$ 
\end{itemize}\end{lemma}

The following result follows immediately from Lemma \ref{RMXE}.
\begin{corol} If $r$ and $r'$ denote the rank functions of $M$ and  $M_X^e$, respectively then $ r'(M_X^e) = r(M)+1$.\end{corol}

\begin{prop} Let $M$ be a matroid on $E$, $X \subset E$, $C\in \mathcal C$$_{OX}$ and $C'\in \mathcal C$$_{EX}$ such that $e\in C \subseteq (A\cup e)$ and $ e\in C' \subseteq (A\cup e)$. Then there is a circuit $ C''\in \mathcal C$$_{OX}$ such that  $ C''\subseteq A$.\end{prop}

\begin{proof} $C\in \mathcal C$$_{OX}$ and $C'\in \mathcal C$$_{EX}$ are circuits of $M$ such that $e\in C \subseteq (A\cup e)$ and $ e\in C' \subseteq (A\cup e)$. Then the symmetric difference $ (C \Delta C')$ contains a circuit $C''$ containing an odd number of elements of $X$. \end{proof}

Let $M$ be a matroid on a set $E$ with rank function $r$, $X \subset E$ and $A\subseteq E$. Then for a subset $A$ of $E$ we define the sets $\mathcal T$ $(A)$ and $ \mathcal F$$(A)$ as follows.
\begin{enumerate}
\item $\mathcal T$ $(A)  =  \{x \in (E - A) ~|~ x\neq e$ and there is a circuit $C\in \mathcal C$$_{OX}$ of $M$ such that $x$, $e \in C$ and $C\subseteq (A\cup e) \cup x \}$.
\item $ \mathcal F$$(A) = \{ x \in (cl(A)- A) ~|~$ there is a circuit $C\in \mathcal C$$_{OX}$ but not in $C$$_{EX}$ such that $x\in C $ and $C \subset cl(A)$.
\end{enumerate}

We have the following proposition.
\begin{prop} If $e\in cl(A)$, then the set $\mathcal T$$(A)$$\subseteq cl(A)$.\end{prop}
\begin{proof} Let $e\in cl(A)$ and $x\in $ $\mathcal T$$(A)$. Then there is a circuit $C$ of the matroid $M$ such that $x\in C $ and $C \subseteq A\cup \{e, x\}$. Consequently, $x\in cl(A\cup  e)$. Since $e\in cl(A)$, $ cl(A \cup  e) = cl(A)$. Thus $x\in cl(A)$. \end{proof}

\section{The Main Theorem}

Throughout this section, we assume that $M$ is a binary matroid on a set $E$ and $M^e_X$ is the splitting matroid of $M$ with respect to the subset $X$ of $E$ and $e\in X$. Further, $A'\subseteq E\cup \{a, \gamma \}$ and $A = A'\setminus \{a, \gamma\}$.

In the following Theorem we characterize the closure operator of the es-splitting matroid $M_X^e$ in terms of the closure operator of the matroid $M$. 
\begin{theorem} \label{thm} Let $M$ be a binary matroid on a set $E$ and $M^e_X$ be the splitting matroid of $M$ with respect to a subset $X$ of $E$ where $e\in X$. Suppose $A'\subseteq E\cup \{a, \gamma \}$ and $A = A'\setminus \{a, \gamma\}$. Then $cl'(A')$ is given by one of the sets $ cl(A) - $$\mathcal F$$(A)$, $ cl(A)$, $ cl(A) \cup a$, $(cl(A)- $ $\mathcal F$ $(A))$ $\cup$ $\gamma$, $(cl(A)- $ $\mathcal F$ $(A))$ $\cup$ $\gamma$ $\cup $ $\mathcal T$$(A)$, $ (cl(A)\cup \gamma)$ $\cup$ $\mathcal T$$(A)$ and 
$cl(A) \cup \{ a, e, \gamma \}$. \end{theorem}

The proof of Theorem \ref{thm} follows from Lemmas \ref{l2}, \ref{l3}, \ref{l4}, \ref{l5}, \ref{l6}, \ref{l7} and \ref{l8}.
\begin{lemma}\label{l2} If $A' = A $ and $(A\cup  e)$ contains no $OX$-circuit. Then $cl'(A') =  cl(A) - $$\mathcal F$$(A)$.\end{lemma}

\begin{proof} Suppose, $A' = A $ and $(A\cup  e)$ contains no $OX$-circuit. Let $x \in cl'(A')$. Then $x \in A' = A$ or $x\in cl'(A')- A$. If $x \in A$,  then we are through. Now if $x\in cl'(A')- A'$, then there is a circuit $C'$ of $M^e_X$ such that $x\in C'$ and $C' \subseteq A'\cup  x = A \cup  x$. 

If $x\in E$, then by Proposition \ref{CMXE}, $C'$ is a circuit of $M$ or $C'$ is the union of two circuits $C_1$ and $C_2$, where $C_1$, $C_2 \in$ $\mathcal C$$_{OX}$. If $C'$ is a circuit of $M$, then $C'$ is an $EX$-circuit. Therefore, $x \in cl(A) - $$\mathcal F$$(A)$. On the other hand, if $C' = C_1 \cup C_2$, then without loss of generality assume that $x\in C_1$. Then $ C_1 \cup C_2 \subseteq A\cup  x$ implies that  $C_2 \subseteq A$, a contradiction.
	
If $x = a$, then $a\in C'$ and $C'$ is a circuit of $M^e_X$ contained in $A\cup a$. This implies that $C'$ = $C\cup  a$ where $C\in$ $\mathcal C$$_{OX}$. Thus, $a\in C' \subseteq A\cup a$ and hence $C \subseteq A$. This is a contradiction to the fact that $A$ contains no member of  $\mathcal C$$_{OX}$. Therefore, $x\neq a$.
	
If $x = \gamma$, then one of the following cases occurs.
\begin{itemize}
\item [(i)] $C' = C \cup \{e, \gamma\} \subseteq A\cup \gamma$ where $e\notin C$ and $C$ contains an odd number of elements of $X$ in $M$. Then $x\in C' = C \cup \{e, \gamma\}\subseteq A\cup \gamma$ and this implies that $C \subseteq A$ but this is a contradiction.
\item [(ii)] $C' = (C \setminus e )\cup \gamma \subseteq A\cup \gamma$ where $C$ contains an odd number of elements of $X$ and $e\in C$. It follows that $e\in C_{OX} \subseteq A \cup  e$; a contradiction.
\item [(iii)] $C' = (C\setminus e) \cup \{a, \gamma \} \subseteq A\cup \gamma$ where $(C\setminus e)$ contains an odd number of elements of $X$ and $e\in C$. This implies that $a\in A$ which is also a contradiction. Therefore, $x \neq \gamma$ and $cl'(A')\subseteq cl(A)- $$\mathcal F$$(A)$.
\end{itemize}

	Conversely, let $x\in cl(A)- $$\mathcal F$$(A)$. If $x \in A$, then there is nothing to prove. If $x \in (cl(A) - A)$ and $x \notin $ $\mathcal F$$(A)$, then there exists an $EX$-circuit $C\in$ $\mathcal C$$_{EX}$ of $M$ such that $x\in C$ and $C \subseteq A \cup  x$. Now $C' = C$ is a circuit of ${M_X^e}$ and $x\in C'\subseteq A \cup  x$. This implies that $ x\in cl'(A')$.  We conclude that $cl(A)- $$\mathcal F$$(A) \subseteq cl'(A')$.\end{proof} 

\noindent We illustrate the above Lemma with the help of following example. 
\begin{example}\label{exm1} Consider the matroid $M =M(G)$ corresponding to the graph shown in the Figure 2. Let $X = \{x, y\}$, $e = y$ and $M_{x,y}^e$ is the corresponding es-splitting matroid. \end{example}
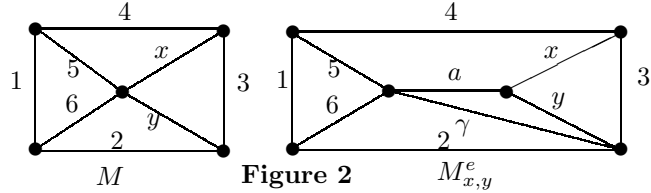
\unitlength .8mm 
\linethickness{0.4pt}
\ifx\plotpoint\undefined\newsavebox{\plotpoint}\fi 
\begin{picture}(107.25,28)(00,00)
\put(7.25,4.25){\framebox(31.25,20.25)[cc]{}}
\put(7.25,24.5){\circle*{2}}
\put(38.5,24.25){\circle*{2}}
\put(7.25,4.5){\circle*{2}}
\put(38.5,4.25){\circle*{2}}
\put(21.75,14){\circle*{2}}
\multiput(7.5,24.5)(.0456730769,-.0336538462){312}{\line(1,0){.0456730769}}
\multiput(21.75,14)(.0579584775,-.0337370242){289}{\line(1,0){.0579584775}}
\multiput(7.25,4.5)(.0509090909,.0336363636){275}{\line(1,0){.0509090909}}
\multiput(21.25,13.75)(.0552884615,.0336538462){312}{\line(1,0){.0552884615}}
\put(50,4.5){\framebox(54.5,19.5)[cc]{}}
\put(50,24){\circle*{2}}
\put(50,4.5){\circle*{2}}
\put(104.5,24){\circle*{2}}
\put(104.5,4.5){\circle*{2}}
\put(66,14.25){\circle*{2}}
\put(85.5,14){\circle*{2}}
\multiput(50,23.75)(.0567375887,-.0336879433){282}{\line(1,0){.0567375887}}
\put(66,14.25){\line(1,0){19.25}}
\put(85.25,14.25){\line(2,1){19}}
\multiput(50,4.75)(.0567375887,.0336879433){282}{\line(1,0){.0567375887}}
\multiput(66,14.25)(.1323529412,-.0337370242){289}{\line(1,0){.1323529412}}
\multiput(104.25,4.5)(-.0657439446,.0337370242){289}{\line(-1,0){.0657439446}}
\put(21,26){4}
\put(3,14.5){1}
\put(19.75,4.50){2}
\put(40.75,14.25){3}
\put(12.5,16.5){5}
\put(12.25,10.75){6}
\put(27,19.75){$x$}
\put(25.75,8.5){$y$}
\put(75.25,26){4}
\put(47.25,14.75){1}
\put(74,4.50){2}
\put(107.25,15.25){3}
\put(55.75,16.5){5}
\put(55.5,10.5){6}
\put(91.75,19.75){$x$}
\put(93,12){$y$}
\put(75.75,16){$a$}
\put(77,7.5){$\gamma$}
\put(19.67,0.1){\makebox(0,0)[cc]{\bf {$M$}}}
\put(78.67,0.1){\makebox(0,0)[cc]{\bf {$M_{x,y}^e$}}}
\put(50.67,00){\makebox(0,0)[cc]{\bf {Figure 2}}}
\end{picture}
\vskip.2in
  
Let $A' = \{4, 5\}$. Then $A' = A$, $cl(A) = \{4, 5, x\}$, $\mathcal F$$(A)= \{x\}$ and $cl'(A') = \{4, 5\} = cl(A) - $$\mathcal F$$(A)$.

\begin{lemma}\label{l3} Suppose that $A' = A$ and $cl(A)$ contains no $OX$-circuit. Then $cl'(A') = cl(A)$. \end{lemma}

\begin{proof} Suppose, $A' = A$ and $cl(A)$ contains no $OX$-circuit. If $x \in cl'(A')$, then $x \in A'$ or $x\in cl'(A')- A'$. If $x \in A'$, then we are through. Now suppose $x\in cl'(A')- A'$ and let $C'$ be a circuit of $M^e_X$ such that $x\in C' \subseteq A'\cup  x$.

If $x\in E$, then by Proposition \ref{CMXE}, $C'$ is a circuit of $M$ or $C'$ is the union of two circuits $C_1$ and $C_2$, belonging to the set $\mathcal C$$_{OX}$. If $C'$ is a circuit of $M$, then $C'$ is an $EX$-circuit. Therefore, $x \in cl(A)$. On the other hand, if $C' = C_1 \cup C_2$, then without loss of generality assume that $x\in C_1$. Then $ C_1 \cup C_2 \subseteq A\cup  x$ implies that  $C_2 \subseteq A$, a contradiction.

If $x = a$, then $a\in C' \subseteq A \cup a$ in $M^e_X$ and this implies that $C'$ = $C\cup  a$ where $C\in$ $\mathcal C$$_{OX}$ is a circuit of $M$. But $a\in C' \subseteq A \cup a$ implies that $C \subseteq A$ and this is a contradiction to the fact that $A$ contains no member of $\mathcal C$$_{OX}$. Therefore, $x\neq a$.
	
If $x = \gamma$, then $\gamma \in C'$ and $C' \subseteq A \cup \gamma$. It follows that $C'$ has one of the following three types of forms.
\begin{itemize}
\item [(i)] $C' = C \cup \{e, \gamma\} \subseteq A \cup \gamma$ where $e\notin C$ and $C\in$ $\mathcal C$$_{OX}$ is a circuit of $M$. Then $C \subseteq A $ and we get a contradiction.
\item [(ii)] $C' = (C \setminus e )\cup \gamma \subseteq A \cup \gamma$ where $C\in$ $\mathcal C$$_{OX}$ is a circuit of $M$ and $e\in C$. Consequently,  $C \subseteq A \cup  e$ and $e\in cl(A)$. This is a contradiction to the fact that $e\notin cl(A)$.
\item [(iii)] $C' = (C\setminus e) \cup \{a, \gamma \} \subseteq A \cup \gamma$ where $(C\setminus e) \in$ $\mathcal C$$_{OX}$ is a circuit of $M$ and $e\in C$. We conclude that $C\subseteq A \cup  e$ and hence $e\in cl(A)$, a contradiction.
\end{itemize}

 Conversely, let $x\in cl(A)$. If $x\in A$, then we are through. If $x\in cl(A) - A $, then there is a circuit $C$ of $M$ such that $x\in C\subseteq A\cup  x$. As $cl(A)$ contains no member of $\mathcal C$$_{OX}$, $C$ contains an even number of elements of $X$. Thus $C$ is also a circuit of $M_X^e$. Thus, $x\in cl'(A')$. This completes the proof of the Lemma. \end{proof}


In order to elaborate the above Lemma, consider $A' = \{1, 5\}$ in Example \ref{exm1}. Then $A' = A$, $cl(A) = \{1, 5, 6\}$ and $cl'(A') = \{1, 5, 6\}$.

\begin{lemma}\label{l4} Suppose $e \notin cl(A)$ and one of the following conditions is true.
\begin{enumerate}
		\item $A' = A$ and $A$ contains an $OX$-circuit.
		\item $A' = A\cup a$.
\end{enumerate} Then $cl'(A') =  cl(A) \cup a$.  \end{lemma} 

\begin{proof} Suppose $e \notin cl(A)$, $A' =A$ and  there is an $OX$-circuit in $A$. Let $x\in cl'(A')- A'$ and $C'$ be a circuit of $M^e_X$ such that $x\in C' \subseteq A'\cup  x$.

If $x\in E$, then by Proposition \ref{CMXE}, $C'$ is a circuit of $M$ or $C'$ is the union of two circuits $C_1$ and $C_2$, from the set $\mathcal C$$_{OX}$. If $C'$ is a circuit of $M$, then $C'$ is an $EX$-circuit. Therefore, $x \in cl(A)$. On the other hand if $C' = C_1 \cup C_2$, then without loss of generality, assume that $x\in C_1$ and $C_1 \subseteq A\cup  x$. Consequently, $x\in cl(A)$.

	If $x = a$, then $x \in cl(A) \cup a$ and if $x = \gamma$, then $\gamma \in C' \subseteq A\cup \gamma$ and $e\notin cl(A)$. This implies that $C'$ has one of the following three forms.
\begin{itemize}
\item [(i)] $C' = C \cup \{e, \gamma\} \subseteq A\cup \gamma$ where $e\notin C$ and $C\in$ $\mathcal C$$_{OX}$ is a circuit of $M$. Then $x\in C' = C \cup \{e, \gamma\}\subseteq A\cup \gamma$. This implies that $e\in A$, a contradiction.
\item [(ii)] $C' = (C \setminus e )\cup \gamma \subseteq A \cup \gamma$ where $C\in$ $\mathcal C$$_{OX}$ is a circuit of $M$ and $e\in C$. Then $e\in cl(A)$ and this leads to a contradiction.
\item [(iii)] $C' = (C\setminus e) \cup \{a, \gamma \} \subseteq A \cup \gamma$ where $(C\setminus e) \in$ $\mathcal C$$_{OX}$ is a circuit of $M$ and $e\in C$. We conclude that $e\in cl(A)$, $a\in A$ and this leads to a contradiction. Therefore, $x \neq \gamma$ and $cl'(A')\subseteq cl(A)\cup  a$.
\end{itemize}

	Conversely, let $x\in cl(A)\cup  a$. If $x = a$, then $a\in cl'(A')$ as $A$ contains an element $C$ of $\mathcal C$$_{OX}$. Moreover, $a\in C' = C \cup  a\subseteq A\cup a$ in ${M^e_X}$.
	If $x\in A$, then we are through. Suppose $x\in cl(A)- A$ and let $C$ be a circuit of $M$ contained in $X$ such that $x\in C\subseteq A\cup  x$. If $C\in$ $\mathcal C$$_{EX}$ is a circuit of $M$, then $C' = C$; otherwise $C' = C\cup  a$ is a circuit of ${M^e_X}$. Further, $x\in C$ implies that $x\in cl'(A\cup a)$ whereas $a\in cl'(A')$ implies $ cl'(A\cup a) = cl'(A')$. Thus $x\in cl'(A')$ and $cl(A) \cup  a \subseteq cl'(A')$, as desired. Part (2) follows by argument similar to one as given in part (1).
\end{proof} 

To illustrate the above Lemma, let $A' = \{1, 4, 6, x\}$ in Example \ref{exm1}. Then $A' = A$, $cl(A) = \{1, 4, 6, x\}$ and $cl'(A') = \{1, 4, 6, x, a\}$. And if $A' = \{1, 6, a\}$ then $A = \{1, 6\}$, $cl(A) = \{1, 5, 6\}$ and $cl'(A') = \{1, 5, 6, a\}$.

\begin{lemma}\label{l5} Let $A' = A$ and $A \cup  e$ contains an $OX$-circuit but $A$ contains no $OX$-circuit. Then $cl'(A') = (cl(A)-$$\mathcal F$$(A))$ $\cup$ $\gamma $.\end{lemma}

\begin{proof} Suppose $A'= A$ and $A \cup  e$ contains an $OX$-circuit but $A$ contains no $OX$-circuit. Let $x \in cl'(A')$. If $x \in A'$  or $x = \gamma $, then we are through. Suppose that $x\in cl'(A')- A'$. Then there is a circuit $C'$ of $M^e_X$ such that $x\in C' \subseteq A'\cup  x $. 

If $x\in E$, then by Proposition \ref{CMXE}, $C'$ is a circuit of $M$ or $C'$ is the union of two circuits $C_1$ and $C_2$, from the set $\mathcal C$$_{OX}$. If $C'$ is a circuit of $M$, then $C'$ is an $EX$-circuit. Therefore, $x \in cl(A)-$$\mathcal F$$(A)$. On the other hand, if $C' = C_1 \cup C_2$, then without loss of generality assume that $x\in C_1$. Then $ C_1 \cup C_2 \subseteq A\cup  x$ implies that  $C_2 \subseteq A$, a contradiction.

If $x = a$, then there is a circuit $C'$ of $M^e_X$ such that $a\in C' \subseteq A \cup a$. This implies that $C'$ = $C\cup  a$ for some $C\in$ $\mathcal C$$_{OX}$. Consequently, $C \subseteq A$. This is a contradiction to the fact that $A$ contains no member of  $\mathcal C$$_{OX}$. Therefore, $x\neq a$.
	
	Conversely, let $x\in cl(A)- $ $(\mathcal F$$(A)$$ \cup \gamma)$. If $x\in (cl(A)- A) -$ $\mathcal F$$(A)$, then there is a circuit $C\in$ $\mathcal C$$_{EX}$ such that $x\in C\subseteq A \cup  x = A' \cup  x$. This implies $ x\in cl'(A')$. If $x = \gamma$ and there is a member $C$ of $\mathcal C$$_{OX}$ such that $e\in C \subseteq A \cup  e$, then $C'$ = $(C\setminus e)\cup \gamma$ is a circuit of $M^e_X$ contained in $A\cup \gamma = A'\cup \gamma$. Therefore, $\gamma\in cl'(A')$. Thus, we conclude that $(cl(A) - $$\mathcal F$$(A)) \cup \gamma \subseteq cl'(A')$. This completes the proof. \end{proof}


To ellaborate the above Lemma, let $A = \{2, 6\}$ in Example \ref{exm1}. Then $A' = A$, $cl(A) = \{2, 6, y\}$, $\mathcal F$$(A)= \{y\}$ and $cl'(A') = \{2, 6, \gamma\} = (cl(A)-$$\mathcal F$$(A))$ $\cup$ $\gamma$.

\begin{lemma}\label{l6} Let $A' = A \cup \gamma$, $e \notin cl(A)$ and $cl(A)$ contains an $OX$-circuit but $A$ contains no $OX$-circuit. Then $cl'(A') =  (cl(A)- $ $\mathcal F$$(A))$ $\cup$ $\gamma$ $\cup $ $\mathcal T$$(A)$.\end{lemma}	

\begin{proof} Let $x\in cl'(A') - (A')$. Then there is a circuit say $C'$ of ${M_X^e}$ such that $x\in C'\subseteq A'\cup  x = A\cup \{\gamma, x\}$.

If $\gamma\in C'$, then $C' \subseteq A \cup \{\gamma, x\}$ and one of the following four cases occurs.
\begin{itemize}
\item [(i)] $C' = C\cup \{e, \gamma\}$ where $C$ is a member of $\mathcal C$$_{OX}$ and $e\notin C$. Then $x\in C' =  C\cup \{e, \gamma\} \subseteq A\cup \{\gamma, x\}$ and this implies that $ x\in (C\cup  e) \subseteq A \cup  x$. Consequently, $ x = e\in cl(A)$ and $C\subseteq A$ but this is a contradiction. 
\item	[(ii)] $C' = (C \setminus e)\cup \{a, \gamma\}$ where $C$ is a circuit of $M$ containing $e$ and $C\setminus e$ is a member of $\mathcal C$$_{OX}$. Then $x\in C' =  (C\setminus e )\cup \{a, \gamma\}$ and $C' \subseteq A\cup \{\gamma, x\}$  therefore, $(C\setminus e) \cup  a \subseteq A\cup  x$. Consequently, $x = a$ and $(C\setminus e) \subseteq A$. As $C \subseteq A \cup  e$, it follows that $C$ is a member of $\mathcal C$$_{EX}$ contained in $cl(A)$. This is a contradiction to the fact that $e\notin cl(A)$.
\item	[(iii)] $C' = (C \setminus e)\cup \gamma$ where $C$ is a circuit of $M$ containing $e$ and $C\in$ $\mathcal C$$_{OX}$. Then $x\in C' =  (C\setminus e )\cup \gamma  \subseteq A\cup \{\gamma, x\}$ and hence $x\in C\setminus e \subseteq A\cup  x$. Thus, $x\in C\subseteq A \cup \{e, x\}$ and $x\in $ $\mathcal T$$(A)$ as desired.
\item	[(iv)] $C' = \{a, e, \gamma\}$. Then $x\in C' = \{a, e, \gamma\} \subseteq A\cup \{\gamma, x\}$ and $x = a$ or $x = e$. In either case, we get a contradiction.
\end{itemize}

	If $\gamma \notin C'$, then $C' \subseteq A\cup  x$ and one of the  following two cases occurs.
\begin{itemize}
\item [(i)] $C'$ is a circuit of  ${M_X^e}$ containing an even number of elements of $X$. Then $C' = C$ or $C' = C_1\cup C_2$ where $C$, $C_1$, $C_2\in$ $\mathcal C$$_{OX}$. Then by argument similar to one as in the proof of Lemma \ref{l2}, we conclude that $x\in cl(A)- $ $\mathcal F$$(A)$. 
\item	[(ii)] $C'$ is a circuit of ${M_X^e}$ containing an odd number of elements of $X$. Then $C' = C \cup  a \subseteq A\cup  x$ where $C\in$ $\mathcal C$$_{OX}$. This implies that $x = a$ and $C \subseteq A$, a contradiction. Therefore, $C'$ contains an even number of elements of $X$. \end{itemize}

Consequently, $cl'(A) \subseteq (cl(A) - $ $\mathcal F$$(A)$) $\cup $ $\gamma$ $\cup $ $ \mathcal T$$(A)$.
	
	Conversely, let $x\in (cl(A) - $ $\mathcal F$$(A)) \cup $ $\gamma$ $\cup$ $\mathcal T$$(A)$. If $x\in (A\cup \gamma)$, then we are through.  In the case $x\in $ $\mathcal T$$(A)$, there is a circuit $C$ of $M$ such that $x\in C \subseteq (A\cup \{e, x\})$. Then $C' = (C\setminus e) \cup \gamma$ is a circuit of ${M_X^e}$ and  $x\in C' \subseteq A\cup \{\gamma, x\}$. We conclude that $ x\in cl'(A\cup \gamma)$.	If $x\in [(cl(A) - $ $\mathcal  F$$(A))\cup \gamma - (A\cup \gamma)]$, then there is a circuit $C\in$ $\mathcal C$$_{EX}$ of $M$ such that $x\in C\subseteq A \cup  x$. Thus, $C' = C$ is a circuit of ${M_X^e}$ and $x\in C'$ of ${M_X^e}$ contained in $A\cup \{\gamma, x\}$. This implies $ x\in cl'(A')$ and we conclude that $(cl(A) - $ $\mathcal F$$(A))$ $\cup $ $\gamma$ $ \subseteq cl'(A')$. \end{proof}

For iilustation of the above Lemma, let $A = \{4, 5, \gamma\}$ in Example \ref{exm1}. Then $A = \{4, 5\}$, $cl(A) = \{4, 5, x\}$, $\mathcal F$$(A)= \{x\}$, $\mathcal T$$(A)= \{3\}$ and $cl'(A') = \{3, 4, 5, \gamma\} = (cl(A)- $ $\mathcal F$$(A))$ $\cup$ $\gamma$ $\cup $ $\mathcal T$$(A) $.

\begin{lemma}\label{l7} Let $A' = A\cup \gamma$, $cl(A)$ contains no $OX$-circuit and $e\notin cl(A)$. Then $cl'(A') =  cl(A) \cup \gamma$ $\cup $ $\mathcal T$$(A)$.\end{lemma}

\begin{proof} Suppose $A' = A\cup \gamma$, $cl(A)$ contains no $OX$-circuit and $e\notin cl(A)$. If $x\in A' = A \cup \gamma$, then $x\in (cl(A)\cup \gamma)$. If $x \in cl'(A') - A'$, then there is a circuit $C'$ of ${M^e_X}$ such that $x\in C'$ and $C' \subseteq A'\cup  x$.
	If $\gamma\notin C'$, then $x\in C'\subseteq A\cup  x$. This implies that $x\in cl(A)$ in $M$.
	If $\gamma\in C'$, then $C'$ has one of the following forms. 
\begin{itemize}	
\item [(i)] $C' = C\cup \{e, \gamma\}$ where $C \in$ $\mathcal C$$_OX$ and $e \notin C$. Then $x\in C' =  C\cup \{e, \gamma\} \subseteq A\cup \{\gamma, x\}$. This implies that $C\subseteq A$ and $x = e$ which is not possible as $A$ contains no member of $\mathcal C$$_{OX}$.
\item [(ii)] $C' = (C \setminus e)\cup \{a, \gamma\}$ where $C$ is a circuit of $M$ such that $e \in C$ and $C\setminus e$ contains an odd number of elements of $X$. Then $x\in C' =  (C\setminus e )\cup \{a, \gamma\} \subseteq A\cup \{\gamma, x\}$ and $(C\setminus e) \cup  a \subseteq A\cup  x$. Therefore, $x = a$ and $(C\setminus e) \subseteq A$. Further, $C \subseteq A \cup  e$ implies that $e\in cl(A)$ and this is a contradiction to the fact that $e\notin cl(A)$.
\item [(iii)] $C' = (C \setminus e)\cup \gamma$ where $e\in C$ and $C\in$ $\mathcal C$$_{OX}$ is a circuit of $M$. Then $x\in C' =  (C\setminus e )\cup \gamma \subseteq A\cup \{\gamma, x\}$ and this implies that $x\in C\setminus e \subseteq A\cup  x$. That is $x\in C\subseteq A \cup \{e, x\}$. We conclude that $x\in $ $\mathcal T$$(A)$.
\item [(iv)] $C' = \{a, e, \gamma\}$. Then $x\in C' = \{a, e, \gamma\} \subseteq A\cup \{\gamma, x\}$ and we conclude that $x = a$ or $x = e$. In either case, we get a contradiction. We conclude that, $ cl'(A')\subseteq cl(A)\cup \gamma$ $\cup $ $\mathcal T$$(A)$.
\end{itemize}

	Conversely, let $x\in cl(A)$ $\cup$ $\gamma$ $\cup$ $\mathcal T$$(A)$. If $x = \gamma$ or $x\in A$, then $x\in cl'(A\cup \gamma)$. If $x\in cl(A)- A$, then there is a circuit say $C$ of $M$ such that $x\in C \subseteq A\cup  x$. As $cl(A)$ contains no member of $\mathcal C$$_{OX}$, $C\in$ $\mathcal C$$_{EX}$. Thus, $C$ is also a circuit of $M_X^e$ and $x\in cl'(A')$.  In the case $x\in $ $\mathcal T$$(A)$, there is a circuit $C$ of $M$ such that $x\in C \subseteq A \cup \{e, x\}$. Then $C' = (C\setminus e) \cup \gamma $ is a circuit of ${M_X^e}$ and  $x\in C' \subseteq A\cup \{\gamma , x\}$. We conclude that $ x\in cl'(A\cup \gamma) = cl'(A')$.\end{proof} 
		
To exemplify the above Lemma, let $A = \{1, 6\}$ in Example \ref{exm1}. Then $A' = A$, $cl(A) = \{1, 5, 6\}$, $\mathcal T$$(A)= \{2\}$ and $cl'(A') = \{1, 2, 5, 6, \gamma\} = cl(A) \cup \gamma$ $\cup $ $\mathcal T$$(A)$.
	
\begin{lemma}\label{l8} Suppose one of the following conditions is satisfied
	\begin{enumerate}
	\item $A' = A\cup \{a, \gamma\}$;
	\item $A' = A\cup a$ and $e\in cl(A)$;
	\item $A' = A\cup \gamma$ and $A$ contains an $OX$-circuit;
	\item $A' = A\cup \gamma$ and $e\in cl(A)$; and
	\item $A' = A$, $A$ contains an $OX$-circuit and $e\in cl(A)$.
\end{enumerate} Then $cl'(A') =  cl(A) \cup \{ a, e, \gamma \}$.\end{lemma}

\begin{proof} Assume that $x\in cl'(A')$. If $x\in A'$ then $x\in A$ or $x\in \{a, \gamma\}$. In either case $x\in cl(A) \cup \{ a, e, \gamma \}$. So, assume that $x\in cl'(A')- A'$. Then there is a circuit $C'$ of $M^e_X$ such that $x\in C' \subseteq A'\cup  x $.
  If $x\in E$, then by Proposition \ref{CMXE}, $C'$ is a circuit of $M$ or $C' = C_1 \cup C_2$, where $C_1 , C_2 \in$ $\mathcal C$$_{OX}$. In either case $x\in cl(A)$.
	Therefore, we conclude that $cl'(A')\subseteq cl(A)\cup \{a, e, \gamma\}$.

	Conversely, let $x\in cl(A)\cup \{a, e, \gamma\}$. If $x = a$, then in Case 1) and 2) $a\in A'$.\\
Case 3) If $C$ is a member of $\mathcal C$$_{OX}$ contained in $A$, then $C' = C \cup  a $ $\subseteq$ $A \cup a$ forms a circuit of $M^e_X$. Thus, $a\in cl'(A')$. \\
Case 4) If $\gamma \in A'$ and $e \in cl(A)$, then there is a $EX$ or $OX$-circuit $C$ containing $e$. Then $C' = (C \setminus e) \cup \{a, \gamma\}$ or  $C' = (C \setminus e) \cup \gamma$ are circuits of $M_X^e$ contained in $A' \cup a = A\cup \{a, \gamma\}$.\\
Case 5) If $A$ contains an $OX$-circuit, then $a \in cl'(A')$. 

Now assume that $x\in cl(A)- A$ and let $C$ be a circuit of $M$ such that $x\in C\subseteq A\cup  x$. If $C\in$ $\mathcal C$$_{EX}$, then $C' = C$; otherwise $C' = C\cup  a$ is a circuit of ${M^e_X}$. Thus, $x\in C$ implies that $x\in cl'(A\cup a)$. As $a\in cl'(A')$, it follows that $x\in cl'(A\cup a) = cl'(A')$.

If $x = \gamma$, then in Cases 1, 3 and 4, $\gamma \in A'$ implies $\gamma \in cl'(A')$. In Cases 2 and 5, $e \in cl(A)$ then one of the following two cases occurs. 
	\begin{itemize}
	\item [(i)] $e\in C\subseteq A \cup  e$ where $C\in$ $\mathcal C$$_{OX}$ and $C  \subseteq cl(A)$. Then $C' = (C\setminus e)$ $\cup$ $\gamma$ $\subseteq A \cup \gamma$ and we conclude that $\gamma \in cl'(A')$. 
	\item [(ii)] $e\in C\subseteq A \cup  e$ where $C\in$ $\mathcal C$$_{EX}$ and $C  \subseteq cl(A)$. Then $C' = ((C\setminus e) \cup \{a, \gamma\})\subseteq A \cup \{a, \gamma \}$ and since $a\in cl'(A')$, $\gamma \in cl'(A\cup a) = cl'(A')$. This implies that $cl(A)\cup \{a, \gamma\} \subseteq cl'(A')$. \end{itemize} 
	
	In all the above cases, we observe that $a, \gamma \in cl'(A')$. Consequently, $\{a, e, \gamma\}$ forms a circuit in $M_X^e$ and hence $e\in cl'(A')$. This completes the proof.
	\end{proof}

We illustrate the above Lemma with the help of Example \ref{exm1} for different types of $A$. 
\begin{enumerate}
	\item If $A' = \{a, \gamma\}$ then $cl'(A') = \{a, e, \gamma\}$.
	\item If $A' = \{a, 6, 2\}$ then $cl'(A') = \{a, 6, 2, e, \gamma\}$.
	\item If $A' = \{4, 5, x, \gamma\}$ then $cl'(A') = \{4, 5, x, a, e, \gamma\}$.
	\item If $A' = \{e, \gamma\}$ then $cl'(A') = \{a, e, \gamma\}$.
	\item If $A' =\{2, 6, y\}$ then $cl'(A') = \{2, 6, a, e, \gamma\}$.
\end{enumerate}

\section {Flats of es-splitting matroids}
In the following Theorem we characterize the flats of the es-splitting matroid $M_X^e$ in terms of the flats of the matroid $M$. 
\begin{theorem} \label{thm} Let $M$ be a binary matroid on a set $E$ and $M^e_X$ be the splitting matroid of $M$ with respect to a subset $X$ of $E$ where $e\in X$. Suppose $A'\subseteq E\cup \{a, \gamma \}$ and $A = A'\setminus \{a, \gamma\}$ is a flat of $M$. Then $A'$ is a flat of $M^e_X$ if one one of the following condition is satisfied 
\begin{enumerate}
\item $A' = A$ and $(A\cup e)$ contains no $OX$-circuit and $\mathcal F$$(A) = \phi$. 
\item $A' = A$ and $cl(A)$ contains no $OX$-circuit.  
\item $A' = A\cup a$, $e \notin cl(A)$.
\item $A' = A \cup \gamma$, $e \notin cl(A)$ and $cl(A)$ contains an $OX$-circuit but $A$ contains no $OX$-circuit, $\mathcal F$ $(A))= \phi$ and $\mathcal T$$(A) = \phi$.
\item $A' = A\cup \gamma$, $cl(A)$ contains no $OX$-circuit and $e\notin cl(A)$ and $\mathcal T$$(A) = \phi$ and 
\item $A' = A\cup \{a, \gamma\}$ and $e \in A$. 
\end{enumerate}
\end{theorem}
The proof of Theorem \ref{thm} follows from Lemmas \ref{l2}, \ref{l3}, \ref{l4}, \ref{l5}, \ref{l6}, \ref{l7} and \ref{l8}. We illustrate the above Theorem with the help of Example \ref{exm1}. The set of all flats of $M$

$\{\{1\}$, 	$\{2\}$, $\{3\}$, $\{4\}$, $\{5\}$,	$\{6\}$,	$\{x\}$,	$\{y\}$,	$\{1, 4\}$,	$\{1, 3\}$, 	$\{1, x\}$,	$\{1, y\}$,	$\{1, 2\}$,	$\{2, 3\}$, $\{2, 4\}$, $\{2, 5\}$, $\{2, x\}$, $\{3, 4\}$, $\{3, 5\}$, $\{3, 6\}$, $\{4, 6\}$, $\{4, y\}$, $\{1, 5, 6\}$, $\{x, y, 3\}$, $\{2, 6 , y\}$, $\{1, 2, 3, 4\}$, $\{4, 5, x\}$, $\{1, 2, 5, 6, y\}$, $\{3, 4, 5, x, y\}$, $\{1, 4, 5, 6, x\}$, $\{2, 3, 6, x, y \}$, $E\}$.

The set of all flats of $M^e_{x, y}$ is 
$\{\{1\}$, $\{2\}$, $\{3\}$, $\{4\}$, $\{5\}$,	$\{6\}$,	$\{x\}$,	$\{y\}$, $\{a\}$, $\{\gamma\}$,	$\{1, 4\}$,	$\{1, 3\}$, 	$\{1, x\}$,	$\{1, y\}$,	$\{1, 2\}$,	$\{2, 3\}$, $\{2, 4\}$, $\{2, 5\}$, $\{2, x\}$, $\{3, 4\}$, $\{3, 5\}$, $\{3, 6\}$, $\{4, 6\}$, $\{4, y\}$, $\{1, a\}$, $\{2, a\}$, $\{3, a\}$, $\{4, a\}$, $\{5, a\}$,	$\{6, a\}$,	$\{x, a\}$,	$\{1, \gamma\}$, $\{3, \gamma\}$, $\{4, \gamma\}$, $\{5, \gamma\}$,	$\{x, \gamma\}$, $\{1, 5, 6\}$, $\{x, y, 3\}$, $\{a, y ,\gamma\}$, $\{2, 6 ,\gamma\}$, $\{1, 2, 3, 4\}$, $\{4, 5, a, x\}$, $\{1, 2, 5, 6, \gamma\}$, $\{ 3, a, x, y, \gamma\}$, $\{2, 6, a, y, \gamma\}$, $\{3, 4, 5, a, x, y\}$,\\ $\{1, 4, 5, 6, a, x\}$, $\{1, 2, 5, 6, a, y, \gamma\}$, $\{2, 3, 6, a, x, y, \gamma\}$, $E\}$.

\end{document}